\documentclass[12pt,reqno]{amsart}
\usepackage{amssymb}
\usepackage{amsfonts}
\usepackage{amssymb,latexsym}
\usepackage{enumerate}
\usepackage{mathrsfs}
\usepackage{url}

\makeatletter
\@namedef{subjclassname@2020}{%
  \textup{2020} Mathematics Subject Classification}
\makeatother

  [2002/01/19 v2.2g %
    AMS font definitions%
  ]
\DeclareFontFamily{U}{euf}{}
\DeclareFontShape{U}{euf}{m}{n}{%
  <5><6><7><8><9>gen*eufm%
  <10><10.95><12><14.4><17.28><20.74><24.88>eufm10%
  }{}
\DeclareFontShape{U}{euf}{b}{n}{%
  <5><6><7><8><9>gen*eufb%
  <10><10.95><12><14.4><17.28><20.74><24.88>eufb10%
  }{}

  [2002/01/19 v2.2g %
    AMS font definitions%
  ]
\DeclareFontFamily{U}{msb}{}
\DeclareFontShape{U}{msb}{m}{n}{%
  <5><6><7><8><9>gen*msbm%
  <10><10.95><12><14.4><17.28><20.74><24.88>msbm10%
  }{}

  [2002/01/19 v2.2g %
    AMS font definitions%
  ]
\DeclareFontFamily{U}{msa}{}
\DeclareFontShape{U}{msa}{m}{n}{%
  <5><6><7><8><9>gen*msam%
  <10><10.95><12><14.4><17.28><20.74><24.88>msam10%
  }{}

\newtheorem{theorem}{Theorem}[section]
\newtheorem{lemma}[theorem]{Lemma}
\newtheorem{proposition}[theorem]{Proposition}

\newtheorem{corollary}[theorem]{Corollary}

\theoremstyle{definition}
\newtheorem{remark}[theorem]{Remark}

\newtheorem{example}[theorem]{Example}

\numberwithin{equation}{section} 

\begin{document}

\def\Re{{\rm Re}}
\def\Im{{\rm Im}}
\def\Li{{\rm Li}}

\title[Sums of infinite series]
{Sums of infinite series involving the Dirichlet lambda function}

\begin{abstract}
The Dirichlet lambda function $\lambda(s)$ is defined for $\mathrm{Re}(s) > 1$ by
\[
\lambda(s) = \sum_{n=0}^{\infty} \frac{1}{(2n+1)^s}.
\]
This function was initially studied by Euler on the real line, where he denoted it by $N(s)$.

In this paper, by applying the partial fraction decomposition of $\pi \tan(\pi x)$ and explicit evaluations of the integrals
\[
\int_0^{\frac{1}{2}} x^{2m-1} \cos(2l\pi x)  dx \quad \text{and} \quad \int_0^{\frac{1}{2}} x^{m-1} \log \cos(\pi x)  dx,
\]
for positive integers $l$ and $m$, we derive closed-form expressions for several classes of infinite series involving $\lambda(s)$. 
We also demonstrate that the values $\lambda(k)$ for even integers $k \geq 2$ arise as constant terms in the Fourier expansions of Eisenstein series associated with the congruence subgroup
\[
\Gamma_0(2) := \left\{ 
\begin{pmatrix} 
    a & b \\ 
    c & d 
\end{pmatrix} \in \operatorname{SL}_2(\mathbb{Z}) : c \equiv 0 \pmod{2}
\right\}.
\]
\end{abstract}

\author{Su Hu}
\address{Department of Mathematics, South China University of Technology, Guangzhou, Guangdong 510640, China}
\email{mahusu@scut.edu.cn}

\author{Min-Soo Kim}
\address{Department of Mathematics Education, Kyungnam University, Changwon, Gyeongnam 51767, Republic of Korea}
\email{mskim@kyungnam.ac.kr}

\subjclass[2000]{11M06}
\keywords{Dirichlet lambda function, Zeta function, Infinite series}


\maketitle

\vspace*{1ex}

\section{Introduction}\label{intro}

Let
\begin{equation}\label{zeta-def}
\zeta(s)=\sum_{n=1}^\infty\frac1{n^s},\quad\textrm{Re}(s)>1
\end{equation} 
be Riemann's zeta functions.
The following sums of infinite series are well-known and may date back to Euler
(see \cite[pp. 479--480, (6) and (8)]{Jo}):
\begin{equation}\label{10}
\sum_{n=1}^\infty(\zeta(2n)-1)=\frac34,
\end{equation}
\begin{equation}\label{11}
\sum_{n=1}^\infty\frac{\zeta(2n)-1}{n}=\log2
\end{equation}
and
\begin{equation}\label{17}
\sum_{n=1}^\infty\frac{\zeta(2n)-1}{n+1}=\frac{3}{2}-\log\pi.
\end{equation}
By solving a problem proposed in  Mathematics Magazine \cite{Omarjee},
Mortini and Rupp \cite{MR} obtained the following formula
\begin{equation}\label{13}
\sum_{n=1}^\infty\frac{\zeta(2n)-1}{n+2}=\frac{7}{4}.
\end{equation}
And they further got  the following more general results.

\begin{theorem}[{\cite[Theorem 1.2]{MR}}]\label{Theorem 1.1}
 Let \( m \in \mathbb{N} \). Then
\begin{equation}\label{3}
\begin{aligned}
&\quad\sum_{n=1}^\infty\frac{\zeta(2n)-1}{n+m}\\
&=\begin{cases}\displaystyle
\frac{1}{2m}+H_m-\log\pi-2m\sum_{j=1}^{m-1}(-1)^{j+1}\binom{2m-1}{2j-1}\frac{(2j-1)!}{(2\pi)^{2j}}\zeta(2j+1)~~&\textrm{if}~~m\geq 2,\\
\displaystyle
\frac{3}{2}-\log\pi~~&\textrm{if}~~m=1,
\end{cases}
\end{aligned}
\end{equation}
where $H_m$ denotes the harmonic numbers:
\begin{equation}\label{Ha}
	H_m=1+\frac12+\cdots+\frac1m=\int_0^1\frac{1-t^m}{1-t}dt \quad\text{and}\quad H_0:=0.
\end{equation}
\end{theorem}

In the following, we denote by $\lceil x \rceil$ the smallest integer greater than or equal to $x$; that is,
\[
\lceil x \rceil = \min\{n \in \mathbb{Z} \mid n \geq x\}.
\]
As is customary, an empty sum is defined to be zero.

\begin{theorem}[{\cite[Theorem 3.2]{MR}}]\label{Theorem 1.2}
 Let \( m \in \mathbb{N} \). Then
\begin{equation}\label{3-1}
\begin{aligned}
&\quad\sum_{n=1}^{\infty} \frac{\zeta(2n)-1}{2n+m} \\
&= \frac{1}{2} \left( \frac{1}{m} - R(m) - \log \pi - m \sum_{j=1}^{\lceil \frac{m}{2} \rceil - 1} (-1)^{j+1} \binom{m-1}{2j-1} \frac{(2j-1)!}{(2\pi)^{2j}} \zeta(2j+1) \right),
\end{aligned}
\end{equation}
where 
\[
R(m) := 2 \int_{0}^{1} \frac{x^{m+1}-x}{1-x^2} \, dx = 
\begin{cases}
-H_{\frac{m}{2}} & \text{if } m \text{ is even}, \\
-H_{\frac{m-1}{2}} - 2\bar{H}_m + \log 4 & \text{if } m \text{ is odd},
\end{cases}
\]  
and $\bar H_m$ denotes the alternating harmonic numbers:
\begin{equation}\label{al-Ha}
\bar H_m=1-\frac12+\cdots+\frac{(-1)^{m-1}}{m}=\int_0^1\frac{1-(-t)^m}{1+t}dt\quad \text{and}
\quad\bar H_0:=0.
\end{equation}
\end{theorem}

Their approaches are mainly based on the following partial fractional decomposition 
of the function $\pi z \cot(\pi z)$ (see, for example, \cite[Lemma 2.1]{MR})
\begin{equation}\label{1}
\pi z\cot(\pi z)=1+\sum_{n=1}^{\infty}\frac{2z^{2}}{z^{2}-n^{2}},\quad z\in\mathbb{C}\setminus\mathbb{Z}
\end{equation}
and the following trigonometric integrals from 0 to 1 (see \cite[Lemma 2.2]{MR}):
\begin{equation}\label{2}
\begin{aligned}
\int_0^1 x^{2m-1} & \cos(2l\pi x)\,dx \\
&= \begin{cases} 
\displaystyle\sum_{j=1}^{m-1} (-1)^{j+1} \binom{2m-1}{2j-1} (2j-1)! \frac{1}{(2\pi l)^{2j}} & \text{if}~ m \geq 2, \\
0 & \text{if}~ m = 1,
\end{cases}
\end{aligned}
\end{equation}
and
\begin{equation}\label{2+}
\int_0^1 x^{2m} \cos(2l\pi x)\,dx = \sum_{j=1}^{m} (-1)^{j+1} \binom{2m}{2j-1} (2j-1)! \frac{1}{(2\pi l)^{2j}}
\end{equation}
for $l,m\in\mathbb{N}$.

In this paper, we consider a dual perspective to the above result. For $\operatorname{Re}(s) > 1$, define Dirichlet's lambda function as
\begin{equation}\label{lam-def}
\lambda(s) = \sum_{n=0}^{\infty} \frac{1}{(2n+1)^{s}}=\frac{1}{2^{s}}\sum_{n=1}^{\infty}\frac{1}{\left(n+\frac{1}{2}\right)^{s}},
\end{equation}
which can be interpreted as summing over intermediate points of the intervals $[n, n+1)$ ($n \in \mathbb{N}$) on the real line, contrasting with the Riemann zeta function $\zeta(s)$ that sums over endpoints.

Historically, as remarked by Varadarajan \cite[p.~59]{Var}, Euler first studied $\lambda(s)$ on the real axis and denoted it by $N(s)$. He also examined its alternating counterpart (see \cite[p.~954, (1.13)]{SK2019} and \cite[p.~70]{Var}),
\begin{equation}\label{beta-def}
\beta(s) = \sum_{n=0}^{\infty} \frac{(-1)^{n}}{(2n+1)^{s}},
\end{equation}
defined for $\operatorname{Re}(s) > 0$, which he denoted by $L(s)$ (now known as the Dirichlet beta function). Notable its special values include
\begin{align*}
\beta(1) &= \frac{\pi}{4} \quad (\text{corresponding to } \tan^{-1}(1)), \\
\beta(2) &= G \quad (\text{Catalan's constant}), \\
\beta(3) &= \frac{\pi^3}{32}.
\end{align*}
In \cite{SK2019}, we have established several infinite families of linear recurrence relations for $\lambda(2n)$ ($n \in \mathbb{N}$) by using generating functions, along with convolution identities connecting values of $\lambda(2n)$ and $\beta(2n+1)$.

In addition, for $\operatorname{Re}(s) > 0$, define the Dirichlet eta function by
\begin{equation}\label{eta}
\eta(s) = \sum_{n=1}^\infty \frac{(-1)^{n-1}}{n^s}.
\end{equation}
Comparing definitions \eqref{lam-def}, \eqref{beta-def} and \eqref{eta}, we have the following relations
\begin{equation}\label{zeta-lam1}
\frac{\lambda(s)}{2^s - 1} = \frac{\eta(s)}{2^s - 2} = \frac{\zeta(s)}{2^s}
\end{equation}
and
\begin{equation}\label{zeta-lam2}
\zeta(s) + \eta(s) = 2\lambda(s)
\end{equation}
(see, for example, \cite{Wiki}).

The remaining parts of this paper are organized as follows.

In Section~\ref{Pre}, Theorem~\ref{thm-FS2} demonstrates that the lambda-values $\lambda(k)$ arise as the constant terms of the Eisenstein series for the congruence subgroup
$$
\Gamma_{0}(2) := \left\{
\begin{pmatrix}
    a & b \\
    c & d
\end{pmatrix} \in \operatorname{SL}_{2}(\mathbb{Z}) \ : \
c \equiv 0 \pmod{2}
\right\}.
$$
This result stems from the partial fraction decomposition of $\pi \tan(\pi x)$.

Similarly, it is well known that the zeta-values $\zeta(k)$ appear as constant terms of the Eisenstein series for the full modular group $\operatorname{SL}_{2}(\mathbb{Z})$, which is derived from the partial fraction decomposition of $\pi x \cot(\pi x)$ (see (\ref{modular})).

In Section~\ref{main results}, we state our main results. These include evaluations of infinite series involving the Dirichlet lambda function $\lambda(s)$: for $p \in \mathbb{N}$,
$$
\sum_{n=1}^\infty \binom{2n}{p} (\lambda(2n) - 1) \quad \text{and} \quad \sum_{n=1}^\infty n^p (\lambda(2n) - 1)
$$
(see Theorems \ref{thm1}, \ref{thm2}, \ref{thm3}, \ref{thm4} and \ref{thm5}).

Section~\ref{Lammas} contains proofs of necessary lemmas for our approach. These include the partial fraction decomposition of $\pi \tan(\pi x)$ and the explicit calculation of the integral
$\int_0^{\frac{1}{2}} x^{2m-1} \cos(2l\pi x)  dx$ for $l, m \in \mathbb{N}$.

Finally, Section~\ref{proofs} provides proofs for the theorems, propositions, and corollaries stated in the preceding sections.

\section{Eisenstein series for $\Gamma_{0}(2)$}\label{Pre}
In this section, we prove that the lambda-values $\lambda(k)$ arise as constant terms of Eisenstein series for the congruence subgroup
\[
\Gamma_{0}(2) := \left\{
\begin{pmatrix}
    a & b \\
    c & d
\end{pmatrix} \in \operatorname{SL}_{2}(\mathbb{Z}) : c \equiv 0 \pmod{2}
\right\}
\]
(see Theorem~\ref{thm-FS2}). 

First, we recall the Fourier expansion of Eisenstein series for the full modular group $\operatorname{SL}_2(\mathbb{Z})$. The modular group consists of integer matrices with determinant 1:
\[
\operatorname{SL}_2(\mathbb{Z}) = \left\{ \begin{pmatrix} a & b \\ c & d \end{pmatrix} \in \operatorname{Mat}_{2\times2}(\mathbb{Z}) : ad - bc = 1 \right\}.
\]
For an even integer $k \geq 4$ and $z \in \mathbb{C}$ with $\operatorname{Im}(z) > 0$, the weight $k$ Eisenstein series is defined by
\begin{equation}
	G_{k}(z) := \sum_{\substack{(m,n) \in \mathbb{Z}^2 \\ (m,n) \neq (0, 0)}} \frac{1}{(m z + n)^{k}}.
\end{equation}
These series admit the Fourier expansion at the cusp $i\infty$:
\begin{equation}\label{modular}
	G_{k}(z) = 2\zeta(k) + \frac{2(-2\pi i)^{k}}{(k-1)!} \sum_{n=1}^\infty \sigma_{k-1}(n)q^n,
\end{equation}
where $q = e^{2\pi iz}$ and $\sigma_{k-1}(n) = \sum_{d\mid n} d^{k-1}$. For even integers $k \geq 4$, $G_{k}(z)$ is a nonzero modular form of weight $k$ for $\operatorname{SL}_2(\mathbb{Z})$ (see \cite[Propositions 2.6 and 2.8]{Kilford}).

Now for a positive integer $N$, we define the following three fundamental congruence subgroups:
\begin{equation*}
	\begin{aligned}
		\Gamma_0(N) &:= \left\{ \begin{pmatrix} a & b \\ c & d \end{pmatrix} \in \operatorname{SL}_2(\mathbb{Z}) : c \equiv 0 \pmod{N} \right\}, \\
		\Gamma_1(N) &:= \left\{ \begin{pmatrix} a & b \\ c & d \end{pmatrix} \in \Gamma_0(N) : a \equiv d \equiv 1 \pmod{N} \right\}, \\
		\Gamma(N) &:= \left\{ \begin{pmatrix} a & b \\ c & d \end{pmatrix} \in \Gamma_1(N) : b \equiv 0 \pmod{N} \right\}.
	\end{aligned}
\end{equation*}
Notably, $\Gamma_0(2)$ and $\Gamma_1(2)$ coincide when $N=2$.
Let $k$ be an integer which is at least 3 (note that $k$ be odd or even).
Then for a point $\mathbf{a} = (a_1, a_2) \in (\mathbb{Z}/N\mathbb{Z})^2$ and $\mathbf{a}\neq (0,0)$, let
\begin{equation}
	G_{k}^{(a_1,a_2)}(z) :=\sum_{\substack{m_1,m_2 \in \mathbb{Z}\\(m_1,m_2) \equiv (a_1,a_2)\!\!\! \pmod{N}}} \frac{1}{(m_1 z + m_2)^{k}}
\end{equation}
be  the level $N$ Eisenstein series of weight $k$ due to Koblitz \cite[Section III.3]{Ko}. These series satisfy (see \cite[Theorem 2.21]{Kilford}):
\begin{itemize}
	\item $G_{k}^{(a_1,a_2) \bmod N} \in \mathcal{M}_{k}(\Gamma(N))$.
	\item $G_{k}^{(0,a_2) \bmod N} \in \mathcal{M}_{k}(\Gamma_1(N))$.
\end{itemize}
Specifying $N=2$ and $(a_1, a_2)=(0,1)$ and denote
$$G^*_{k}(z):=G_{k}^{(0,1)}(z)=\sum_{\substack{m,n\in\mathbb Z \\ n\,{\rm odd}}}\frac{1}{(2mz+n)^{k}}$$
for $2<k\in\mathbb N,$
we have $G^*_{k}(z) \in\mathcal{M}_{k}(\Gamma_0(2))=\mathcal{M}_{k}(\Gamma_1(2))$.
The following proposition shows its Fourier expansion at the cusp $i\infty$.

(Note. It is easy to check that $G^*_{k}(z)$ approaches a finite limit as $z\to i\infty$:
$$\lim_{z\to i\infty}\sum_{\substack{m,n\in\mathbb Z \\ n\,{\rm odd}}}\frac{1}{(2mz+n)^{k}}
=\sum_{n\in\mathbb Z,n\,{\rm odd}}\frac{1}{n^{k}}=2\lambda(k)$$
for a positive integer $k>2.$)

\begin{theorem}\label{thm-FS2}
	For an even integer $k>2$ and $z \in \mathbb{C}$ with $\operatorname{Im}(z) > 0$, we have
	$$G^*_{k}(z)=2\lambda(k)+\frac{2(\pi i)^{k}}{(k-1)!}\sum_{n=1}^\infty\sigma^*_{k-1}(n)e^{2\pi inz},$$
	where  $\sigma^*_{k-1}(n)$ denotes the alternating divisor functions:
	$$\sigma^*_{k-1}(n)=\sum_{d\mid n}(-1)^dd^{k-1}.$$
\end{theorem}
To establish this result, we require several propositions. We begin by stating the partial fraction decompositions of $\pi \tan(\pi z)$ and $\pi z \cot(\pi z)$, along with the infinite product representation of $\cos(\pi z)$.
\begin{proposition}\label{pro-FS}
	\begin{enumerate}
		\item[\rm(1)]  For all $2z\in\mathbb C\setminus 2\mathbb Z+1=\left\{2m+1\mid m\in\mathbb Z\right\},$ we have
		$$-\pi\tan(\pi z)=\sum_{n=1, n {\rm\,odd}}^\infty\frac{8z}{(2z)^2-n^2}=2\sum_{n=1, n {\rm\,odd}}^\infty\left(\frac1{2z+n}+\frac1{2z-n}\right).$$
		\item[\rm(2)] For all $z\in\mathbb C,$ we have
		$$\cos(\pi z)=\prod_{n=1, n {\rm\,odd}}\left(1-\frac{z^2}{\left(\frac n2\right)^2}\right).$$
	\end{enumerate}
\end{proposition}
\begin{proof}
	They are follows from \cite[p. 43, 1.421(1)]{GR} and \cite[p. 291, (1.15)]{SK2025}.
\end{proof}

By Proposition \ref{pro-FS}(1), we have the following Fourier expansion.

\begin{proposition}\label{cor-FS}
	Let $z \in \mathbb{C}$ with $\operatorname{Im}(z) > 0.$
	Then for any positive integer $k\geq2,$ we have
	$$\sum_{n\in\mathbb Z, n {\rm\,odd}}\frac{1}{\left(z+\frac n2\right)^{k}} 
	=\frac{(-2\pi i)^{k}}{(k-1)!}\sum_{m=1}^\infty(-1)^mm^{k-1}e^{2\pi imz}.$$
\end{proposition}
\begin{proof}
As observed by Vignat (personal communication), this result also follows from the Lipschitz summation formula (see Theorem~1 in Knop and Robins~\cite{KR}). For completeness and accessibility, we provide here a self-contained proof based on the Fourier expansion of $\tan(\pi z)$.

	Let $q = e^{2\pi i z}$ with $|q| = e^{-2\pi \operatorname{Im}(z)} < 1$. Recall the following identity for the tangent function expressed in terms of $q$:
	\begin{equation*}
		\tan(\pi z) 
		= \frac{\sin(\pi z)}{\cos(\pi z)} 
		= i \frac{1 - q}{1 + q} 
		= i \left( 1 + 2 \sum_{m=1}^\infty (-1)^m q^m \right).
	\end{equation*}
	Applying Proposition \ref{pro-FS} to this expansion, we have
	\begin{equation*}
		-\pi i \left( 1 + 2 \sum_{m=1}^\infty (-1)^m q^m \right)
		= 2 \sum_{\substack{n=1 \\ n \text{ odd}}}^\infty \left( \frac{1}{2z + n} + \frac{1}{2z - n} \right).
	\end{equation*}
	To generalize this result, consider $(k-1)$-fold differentiation (justified by uniform convergence on compact subsets of the upper half-plane). For $k \geq 2$, this yields
	\begin{equation*}
		-(2\pi i)^k \sum_{m=1}^\infty (-1)^m m^{k-1} q^m
		= (-1)^{k-1} (k-1)! \sum_{\substack{n \in \mathbb{Z} \\ n \text{ odd}}}^\infty \frac{1}{\left(z + \frac{n}{2} \right)^k},
	\end{equation*}
	where the differentiation operator acts on both sides of the original equality. This completes the proof of the proposition.
\end{proof}

Let $E_{n}(x)$ be  the $n$-th Euler polynomial generated by the Taylor series expansion
	$$
	\sum_{n=0}^\infty E_n(x)\frac{t^n}{n!}=\frac{2e^{xt}}{e^t+1}, 
	$$
	where $|t|<\pi,$ see for example, \cite[p. xxxi]{GR}, \cite[Chapter 20]{OMS} and \cite[p. 86, (40)]{SC}. We have the following result on the special values of Dirichlet's lambda functions.
\begin{proposition}\label{proFS-co2}
	For $k\in\mathbb N,$ we have
	$$\sum_{n\in\mathbb Z,n\,{\rm odd}}\frac{1}{n^{k+1}}
	=\begin{cases}
		\frac{(-1)^{\frac{k+1}{2}}\pi^{k+1}}{2(k!)}E_{k}(0) &\text{if  $k$ is odd}, \\
		0 &\text{if $k$ is even}.
	\end{cases}$$
	In particular, if $k$ is odd, then $$\lambda(k+1)=\frac{(-1)^{\frac{k+1}{2}}\pi^{k+1}}{4(k!)}E_{k}(0).$$
\end{proposition}
\begin{proof}As noted by Vignat (personal communication), this result can also be derived from identity 24.8.5 in the \textit{NIST Handbook of Mathematical Functions} \cite{NIST}, which gives the Fourier expansion of the Riemann zeta function. For completeness, we provide here an alternative proof based on the partial fraction decomposition of $\pi \tan(\pi z)$.

For $2z \in \mathbb{C} \setminus (2\mathbb{Z} + 1)$, the partial fraction decomposition (Proposition \ref{pro-FS}(1) and \cite[p.~43, 1.421(1)]{GR}) shows that 
\[
-\pi \tan(\pi z) = 2 \sum_{\substack{n=1 \\ n \text{ odd}}}^\infty \left( \frac{1}{2z + n} + \frac{1}{2z - n} \right).
\]
Differentiating $k$ times with respect to $z$ (justified by uniform convergence) yields 
\begin{equation}
\left( \frac{d}{dz} \right)^k \left( \pi \tan(\pi z) \right) 
= (-1)^{k+1} k! \sum_{\substack{n \in \mathbb{Z} \\ n \text{ odd}}} \frac{1}{\left( z + \frac{n}{2} \right)^{k+1}}
\end{equation}
for $k \geq 1$.

Conversely, the Taylor series expansion at $z=0$ \cite[p.~371, (18.7)]{Ri} is
\[
\tan(\pi z) = \sum_{n=1}^\infty \frac{(-1)^{\frac{n+1}{2}} 2^n E_n(0)}{n!} (\pi z)^n.
\]
Then define  $f(z) = \pi \tan(\pi z)$. Evaluating the $k$-th derivative at $z=0$ gives
\[
f^{(k)}(0) = 
\begin{cases} 
\pi^{k+1} (-1)^{\frac{k+1}{2}} 2^k E_k(0) & \text{if $k$ is odd}, \\ 
0 & \text{if $k$ is even},
\end{cases}
\]
where $f^{(k)}$ denotes the $k$th derivative. Combining these results completes the proof.
\end{proof}

\subsection*{Proof of Theorem \ref{thm-FS2}}
	By multiplying both sides of $G^*_{k}(z)$ by $2^{k}$, we have the following reformulation:
	\begin{equation*}
		\begin{aligned}
			2^{k}G^*_{k}(z)
			&= \sum_{m\in\mathbb{Z}} \sum_{\substack{n\in\mathbb{Z} \\ n \text{ odd}}} \frac{1}{\left(mz + \frac{n}{2}\right)^{k}} \\
			&= \underbrace{\sum_{\substack{n\in\mathbb{Z} \\ n \text{ odd}}} \frac{1}{\left(\frac{n}{2}\right)^{k}}}_{\text{constant term}} 
			+ \sum_{m\in\mathbb{Z}\setminus\{0\}} \sum_{\substack{n\in\mathbb{Z} \\ n \text{ odd}}} \frac{1}{\left(mz + \frac{n}{2}\right)^{k}} \\
			&= 2\sum_{n=1}^\infty \frac{1}{\left(\frac{2n-1}{2}\right)^{k}} 
			+ \sum_{m=1}^\infty \sum_{\substack{n\in\mathbb{Z} \\ n \text{ odd}}} \left[ \frac{1}{\left(\frac{n}{2} + mz\right)^{k}} + \frac{1}{\left(\frac{n}{2} - mz\right)^{k}} \right] \\
			&= 2^{k+1}\sum_{n=1}^\infty \frac{1}{(2n-1)^{k}} 
			+ 2\sum_{m=1}^\infty \sum_{\substack{n\in\mathbb{Z} \\ n \text{ odd}}} \frac{1}{\left(\frac{n}{2} + mz\right)^{k}}.
		\end{aligned}
	\end{equation*}
	Then applying Proposition \ref{cor-FS} with substitutions $\tau \mapsto m\tau$
	for an even integer $k>2,$ we get
	\begin{equation*}
		\begin{aligned}
			2^{k}G^*_{k}(z)
			&= 2^{k+1}\lambda(k) + \frac{2(2\pi i)^{k}}{(2k-1)!} \sum_{m=1}^\infty \sum_{n=1}^\infty (-1)^n n^{k-1} e^{2\pi i mnz} \\
			&= 2^{k+1}\lambda(k) + \frac{2(2\pi i)^{k}}{(2k-1)!} \sum_{n=1}^\infty \left( \sum_{d \mid n} (-1)^d d^{k-1} \right) e^{2\pi i nz},
		\end{aligned}
	\end{equation*}
	which yields the desired conclusion.

\section{Statement of main results}\label{main results}
In this section, we present our main results on closed-form expressions for infinite series involving the Dirichlet lambda function $\lambda(s)$ (see Theorems \ref{thm1}, \ref{thm2}, \ref{thm3}, \ref{thm4} and \ref{thm5}). Their proofs will be given in Section~\ref{proofs}.

First, we state analogues for Dirichlet's lambda function $\lambda(s)$ of Euler's classical results (\ref{10}), (\ref{11}), and (\ref{17}). (See Section~\ref{proofs} for the proof).
\begin{proposition}\label{pro1}
	\begin{enumerate}
		\item[\rm(1)]  $\displaystyle\sum_{n=1}^\infty(\lambda(2n)-1)=\frac14.$
		\item[\rm(2)] $\displaystyle\sum_{n=1}^\infty\frac{\lambda(2n)-1}n=\log4-\log\pi.$
	\end{enumerate}
\end{proposition}

Furthermore, a ‘dual’ of (\ref{1}) might be the following partial fractional decomposition of $\pi\tan(\pi z)$
(see Proposition~\ref{pro-FS} (1)):
\begin{equation}\label{8}
-\pi\tan(\pi z)=\sum_{n=1, n {\rm\,odd}}^\infty\frac{8z}{(2z)^2-n^2}=2\sum_{n=1, n {\rm\,odd}}^\infty\left(\frac1{2z+n}+\frac1{2z-n}\right),
\end{equation}
for all $2z\in\mathbb C\setminus 2\mathbb Z+1=\left\{2m+1\mid m\in\mathbb Z\right\}$,  
while a ‘dual’ of (\ref{2}) may be considered as the following trigonometric integrals from $0$ to $\frac12$ (see Lemma  \ref{lem1-thm1} and \cite[Theorem 2.3]{SK2025}, respectively):
\begin{equation}\label{9}
\begin{aligned}
\int_0^{\frac12}x^{2m-1}\cos(2l\pi x) dx&=\frac{(-1)^l}{2^{2m}}\sum_{j=1}^{m-1}(-1)^{j-1}\binom{2m-1}{2j-1}(2j-1)!\frac{1}{(l\pi)^{2j}} \\
&\quad+(1-(-1)^l)\frac{(-1)^m(2m-1)!}{(2l\pi)^{2m}}
\end{aligned}
\end{equation}
and
\begin{equation}\label{9+}
\begin{aligned}
m\int_0^{\frac12}x^{m-1}\log\cos(\pi x)dx
&=-\frac{\log2}{2^m}-\frac{m}{2^m}\sum_{j=0}^{m-1}j!\binom{m-1}{j}\frac{\sin\left(\frac{j\pi}{2}\right)}{\pi^{j+1}}\zeta(j+2) \\
&\quad-\frac{m!}{(2\pi)^m}\zeta_E(m+1)\cdot
\begin{cases}
\displaystyle
0 &\text{if $m$ is odd}, \\
\displaystyle
(-1)^{\frac m2+1} &\text{if $m$ is even}.
\end{cases} 
\end{aligned}
\end{equation}

In this paper, by applying (\ref{8}), (\ref{9}) and (\ref{9+}), we shall determine  the following sums of infinite series involving the Dirichlet lambda function.
They are the analogues of Theorems \ref{Theorem 1.1} and \ref{Theorem 1.2} above.

\begin{theorem}\label{thm1}
For $m\in\mathbb N,$ we have
$$\begin{aligned}
\sum_{n=1}^\infty\frac{\lambda(2n)-1}{n+m}
&=\log\left(\frac2\pi\right)+H_m+\frac{2(-1)^m(2m)!\lambda(2m+1)}{\pi^{2m}} \\
&\quad+2m\sum_{j=1}^{m-1}(-1)^j\binom{2m-1}{2j-1}\frac{(2j-1)!2^{2j+1}}{\pi^{2j}(2^{2j+1}-1)}\lambda(2j+1),
\end{aligned}$$
where $H_m$ is given by equation (\ref{Ha}). 
\end{theorem}

\begin{theorem}\label{thm2}
For $m\in\mathbb N,$ we have
$$\begin{aligned}
\sum_{n=1}^\infty\frac{\lambda(2n)-1}{2n+m}
&=\frac12\log\left(\frac2\pi\right)+\frac{m}2\sum_{j=1}^{\left\lceil \frac{m-1}{2}\right\rceil }(-1)^j\binom{m-1}{2j-1}
     \frac{(2j-1)!2^{2j+1}}{\pi^{2j}(2^{2j+1}-1)}\lambda(2j+1) \\
&\quad-\frac{m!(2^{m+1}-2)}{2\pi^{m}(2^{m+1}-1)}\lambda(m+1)
                                      \cdot\begin{cases}
                                      (-1)^{\frac m2+1}&\text{if $m$ is even} \\
                                      0&\text{if $m$ is odd}
                                      \end{cases} \\
&\quad+\begin{cases}
               \frac12 H_{\frac m2}&\text{if $m$ is even}, \\
               \frac12H_{\frac{m-1}{2}}-\log2+\bar H_m&\text{if $m$ is odd},
               \end{cases}
\end{aligned}$$
where $\bar H_m$  is given by equation (\ref{al-Ha}).
\end{theorem}

From Theorems \ref{thm1} and \ref{thm2} we immediately get the following  explicit calculations, respectively.

\begin{example}\label{cor1}
\begin{enumerate}
\item[\rm(1)] $\displaystyle\sum_{n=1}^\infty\frac{\lambda(2n)-1}{n+1}=\log\left(\frac2\pi\right)+1-\frac4{\pi^2}\lambda(3).$
\item[\rm(2)] $\displaystyle\sum_{n=1}^\infty\frac{\lambda(2n)-1}{n+2}
=\log\left(\frac2\pi\right)+\frac{3}2-\frac{96}{7\pi^2}\lambda(3)+\frac{48}{\pi^4}\lambda(5).$
\item[\rm(3)] $\displaystyle\sum_{n=1}^\infty\frac{\lambda(2n)-1}{n+3}
=\log\left(\frac2\pi\right)+\frac{11}6-\frac{240}{7\pi^2}\lambda(3)+\frac{11520}{31\pi^4}\lambda(5)-\frac{1440}{\pi^6}\lambda(7).$
\end{enumerate}
\end{example}

\begin{example}\label{cor2}
\begin{enumerate}
\item[\rm(1)] $\displaystyle\sum_{n=1}^\infty\frac{\lambda(2n)-1}{2n+1}
=1-\frac12\log(2\pi).$
\item[\rm(2)] $\displaystyle\sum_{n=1}^\infty\frac{\lambda(2n)-1}{2n+2}
=\frac12\log\left(\frac2\pi\right)+\frac{1}2-\frac{2}{\pi^2}\lambda(3).$
\item[\rm(3)] $\displaystyle\sum_{n=1}^\infty\frac{\lambda(2n)-1}{2n+3}
=\frac43-\frac12\log(2\pi)-\frac{24}{7\pi^2}\lambda(3).$
\end{enumerate}
\end{example}

Then by combining Theorems \ref{thm1} and \ref{thm2} with Theorems 1.2 and 3.2 of \cite[p. 164 and p. 170]{MR} (see also (\ref{3}) and (\ref{3-1})),
we have the following corollaries.

\begin{corollary}\label{cor-SC-1}
For $m\in\mathbb N,$ we have
$$
\begin{aligned}
\sum_{n=1}^{\infty} \frac{\zeta(2n)}{2^{2n}(n+m)} 
&=\frac1{2m}-\log2-\frac{(-1)^{m}(2m)!(2^{2m+1}-1)}{(2\pi)^{2m}}\zeta(2m+1)  \\
&\quad+\sum_{j=1}^{m-1}(-1)^j\binom{2m}{2j}\frac{(2j)!}{(2\pi)^{2j}}(1-2^{2j})\zeta(2j+1).
\end{aligned}
$$
\end{corollary}

\begin{corollary}\label{cor-SC-2}
For $m\in\mathbb N,$ we have
$$
\begin{aligned}
\sum_{n=1}^{\infty} \frac{\zeta(2n)}{2^{2n}(2n+2m+1)} 
&=\frac1{2(2m+1)}-\frac12\log2   \\
&\quad+\frac{2m+1}{2}\sum_{j=1}^{m}(-1)^j\binom{2m}{2j-1}\frac{(2j-1)!(1-2^{2j})}{(2\pi)^{2j}}\zeta(2j+1).
\end{aligned}
$$
In partcular, 
$$\sum_{n=1}^{\infty} \frac{\zeta(2n)}{2^{2n}(2n+1)} =\frac12-\frac12\log2.$$
\end{corollary}

\begin{remark}\label{re-SC}
(1) Setting $m=1$ in Corollary \ref{cor-SC-1}, we have the following formula recently established by Srivastava and Choi \cite[p.~318, (532)]{SC}:
\[
\sum_{n=1}^{\infty} \frac{\zeta(2n)}{2^{2n}(n+1)} 
= \frac{1}{2} - \log 2 + \frac{7}{2\pi^2}\zeta(3)
= \frac{1}{2} + \log(2^{-1} \cdot B^{14}),
\]
where $B$ satisfies $\log B = \frac{\zeta(3)}{4\pi^2}$ as defined in \cite[p.~168, (29)]{SC}.

\vspace{1mm}
\noindent
(2) Direct verification shows that Corollaries \ref{cor-SC-1} and \ref{cor-SC-2} align with series representations in \cite[(3.15) and (3.20)]{Sri} (pp.~340--341), originally due to Srivastava. See also \cite[pp.~424--425, (8) and (12)]{SC}.
\end{remark}

In the subsequent work \cite{MR2}, the authors have evaluated the series
\[
\sum_{n=1}^\infty \binom{2n}{p} (\zeta(2n) - 1) \quad \text{and} \quad \sum_{n=1}^\infty n^p (\zeta(2n) - 1),
\]
we now establish analogous evaluations for the Dirichlet lambda function $\lambda(s)$. 

Consider the function defined in Lemma \ref{mr-lem}(1):
\begin{equation}\label{ma-s-re}
f(z) := 
\begin{cases} 
\displaystyle
\frac{1}{2}\pi z \tan(\pi z) - \frac{(2z)^2}{1 - (2z)^2} & \text{if $|z| < \frac{3}{2},\  z \not\in \left\{-\frac{1}{2}, 0, \frac{1}{2}\right\} $}, \\[10pt]
\displaystyle
\frac{1}{4} & \text{if $z = \pm \frac{1}{2} $}, \\[10pt]
0 & \text{if $z = 0 $}.
\end{cases}
\end{equation}
The following theorem establishes its power series expansion about $z = \frac{1}{2}$.
\begin{theorem}\label{thm3}
	Setting $\delta_n=\frac{1-(-1)^n}{2}$. If $\left|z-\frac12\right|<1,$  then we have 
	$$f(z)=\sum_{n=0}^\infty c_n \left(z-\frac12\right)^n,$$
	where 
	\begin{equation}\label{thm3.1}
		\begin{aligned}
			c_0:=\frac14, \quad c_n
			&:=\frac1{2^{\delta_n}}\left(
			\frac{2^{2\left\lceil \frac{n}{2}\right\rceil}}{2^{2\left\lceil \frac{n}{2}\right\rceil}-1} \right)
			\lambda\left(2\left\lceil \frac{n}{2}\right\rceil\right)
			-\frac{(-1)^n}{4}
		\end{aligned}
	\end{equation}
	and
	$$f^{(n)}\left(\frac12\right)=n!c_n.$$
\end{theorem}

\begin{remark}\label{rem-g}
	In \cite[Theorem 2.1]{MR2}, the authors have considered the following function
	\[
	g(z) := 
	\begin{cases} 
		\frac{1}{2}(1 - \pi z \cot(\pi z)) - \frac{z^2}{1 - z^2}  & \text{if } |z| < 2,\  z\notin \{-1, 0, 1\}, \\ 
		\frac{3}{4} & \text{if } z = \pm 1, \\ 
		0 & \text{if } z = 0,
	\end{cases}
	\]
	and got  the following power series expansion around $z=1$:
	\[
	g(z) = \sum_{n=0}^{\infty} a_n(z - 1)^n,
	\]
	where \( a_0 := \frac{3}{4} \) and
\begin{equation}\label{co-ap}
	a_n := 
	\begin{cases} 
		\zeta(n+1) + \frac{1}{2^{n+2}} & \text{if } n \text{ is odd}, \\ 
		\zeta(n) - \frac{1}{2^{n+2}} & \text{if } n \text{ is even}.
	\end{cases}
\end{equation}
	Moreover, the \( n \)-th derivative satisfies \( g^{(n)}(1) = n!\, a_n \) (see \cite[p. 937, Theorem 2.1]{MR2}).
\end{remark}

Now we are at the position to calculate the sums of the infinite series
$$\sum_{n=1}^\infty \binom{2n}{p}(\lambda(2n)-1).$$

\begin{theorem}\label{thm4}
	For $p\in\mathbb N,$ we have
	\[
	\sum_{n=1}^\infty \binom{2n}{p} (\lambda(2n) - 1) = \frac{1}{2^{2\left\lceil \frac{p}{2}\right\rceil} - 1} \lambda\left(2\left\lceil \frac{p}{2}\right\rceil\right) - \frac{(-1)^p}{2^{p+2}}.
	\]
\end{theorem}

\begin{example}
	For the first few values of $p,$ we have the following explicit calculations: 
	\begin{enumerate}
		\item[\rm(1)]
		$\displaystyle\sum_{n=1}^\infty n(\lambda(2n)-1)=\frac1{16}+\frac{\pi^2}{48}.$
		\item[\rm(2)]
		$\displaystyle\sum_{n=1}^\infty n(2n-1)(\lambda(2n)-1)=-\frac1{16}+\frac{\pi^2}{24}.$
		\item[\rm(3)]
		$\displaystyle\sum_{n=1}^\infty n(2n-1)(2n-2)(\lambda(2n)-1)=\frac3{32}+\frac{\pi^4}{480}.$
	\end{enumerate}
\end{example}

The above result implies:
\begin{corollary}\label{coro2}
	For $p\in\mathbb N,$ we have
	$$
	\sum_{n=1}^\infty \binom{2n}{p}\frac{\zeta(2n)}{2^{2n}}
	=\left(1-\frac1{2^{2\left\lceil \frac{p}{2}\right\rceil}}\right) \zeta\left(2\left\lceil \frac{p}{2}\right\rceil\right).
	$$
\end{corollary}

\begin{example}Corollary \ref{coro2} provide  explicit results for $\sum_{n=1}^\infty \binom{2n}{p}\frac{\zeta(2n)}{2^{2n}},$ where $p=1,2,3:$
	\begin{enumerate}
		\item[\rm(1)]
		$\displaystyle\sum_{n=1}^\infty n\frac{\zeta(2n)}{2^{2n}}=\frac{\pi^2}{16}.$
		\item[\rm(2)]
		$\displaystyle\sum_{n=1}^\infty n(2n-1)\frac{\zeta(2n)}{2^{2n}}=\frac{\pi^2}{8}.$
		\item[\rm(3)]
		$\displaystyle\sum_{n=1}^\infty n(2n-1)(2n-2)\frac{\zeta(2n)}{2^{2n}}=\frac{\pi^{4}}{32}.$
	\end{enumerate}
\end{example}

Let ${p \brace j}$ denote the Stirling numbers of the second kind, defined by (cf. \cite[p.~116, (9.11)]{QG})
\[
x^p = \sum_{j=0}^p \binom{x}{j} j! {p \brace j}.
\]
These numbers satisfy the recurrence relations:
\begin{equation}
\begin{aligned}
{p \brace j} &= {p-1 \brace j-1} + j{p-1 \brace j}, \\
{p \brace 1} &= 1.
\end{aligned}
\end{equation}

The following theorem evaluates the infinite series
\[
\sum_{n=1}^\infty n^p (\lambda(2n) - 1).
\]
\begin{theorem}\label{thm5}
	For $p\in\mathbb N,$ we have
	$$\sum_{n=1}^\infty n^p(\lambda(2n)-1)
	=\sum_{j=1}^p{p\brace j}j! 
	\left(\frac1{2^{2\left\lceil \frac{j}{2}\right\rceil}-1} \lambda\left(2\left\lceil \frac{j}{2}\right\rceil\right)-\frac{(-1)^j}{2^{j+2}}\right).$$
	\end{theorem}
\begin{example}
	For the first  few values of $p,$ we have the following explicit calculations: 
	\begin{enumerate}
		\item[\rm(1)]
		$\displaystyle\sum_{n=1}^\infty n(\lambda(2n)-1)=\frac1{16}+\frac{\pi^2}{48}.$
		\item[\rm(2)]
		$\displaystyle\sum_{n=1}^\infty n^2(\lambda(2n)-1)=\frac{\pi^2}{32}.$
		\item[\rm(3)]
		$\displaystyle\sum_{n=1}^\infty n^3(\lambda(2n)-1)=-\frac1{128}+\frac{7\pi^2}{192}+\frac{\pi^4}{1920}.$
	\end{enumerate}
\end{example}

The above result implies:
\begin{corollary}\label{coro3}
	For $p\in\mathbb N,$ we have
	$$
	\sum_{n=1}^\infty n^p\frac{\zeta(2n)}{2^{2n}}
	=\frac1{2^p}\sum_{j=1}^p{p\brace j}j! \left(1-\frac1{2^{2\left\lceil \frac{j}{2}\right\rceil}}\right) \zeta\left(2\left\lceil \frac{j}{2}\right\rceil\right).
	$$
\end{corollary}

\begin{example}Corollary \ref{coro3} provide  explicit results for $\sum_{n=1}^\infty n^p\frac{\zeta(2n)}{2^{2n}},$ where $p=1,2,3:$
	\begin{enumerate}
		\item[\rm(1)]
		$\displaystyle\sum_{n=1}^\infty n\frac{\zeta(2n)}{2^{2n}}=\frac{\pi^2}{16}.$
		\item[\rm(2)]
		$\displaystyle\sum_{n=1}^\infty n^2\frac{\zeta(2n)}{2^{2n}}=\frac{3\pi^2}{32}.$
		\item[\rm(3)]
		$\displaystyle\sum_{n=1}^\infty n^3\frac{\zeta(2n)}{2^{2n}}=\frac{7\pi^{2}}{64}+\frac{\pi^{4}}{128}.$
	\end{enumerate}
\end{example}

\section{Lemmas}\label{Lammas}

To prove our main results, we require the following four lemmas.
\begin{lemma} \label{lem1-thm1}
Let $l,m\in\mathbb N.$ Then
$$
\begin{aligned}
\int_0^{\frac12}x^{2m-1}\cos(2l\pi x) dx&=\frac{(-1)^l}{2^{2m}}\sum_{j=1}^{m-1}(-1)^{j-1}\binom{2m-1}{2j-1}(2j-1)!\frac{1}{(l\pi)^{2j}} \\
&\quad+(1-(-1)^l)\frac{(-1)^m(2m-1)!}{(2l\pi)^{2m}}.
\end{aligned}
$$
\end{lemma}
\begin{proof}
Using the integral formula from \cite[Lemma 4.1, p. 301]{SK2025}, we derive the following identity:
\begin{equation*}
\int_0^{\frac{1}{2}} x^n \cos(ax) dx = \sum_{j=0}^n j! \binom{n}{j} \frac{2^{j-n}}{a^{j+1}} \sin\left(\frac{a}{2} + \frac{j\pi}{2}\right) 
- n! \frac{1}{a^{n+1}} \sin\left(\frac{n\pi}{2}\right).
\end{equation*}
This integral representation can also be found in \cite[2.633(2), p. 226]{GR}.
Then substituting $n = 2m-1$ and $a = 2\pi l$ into the above formula, we have
\begin{equation*}
\begin{aligned}
\int_0^{\frac{1}{2}} x^{2m-1} \cos(2l\pi x) dx 
&= \frac{1}{2^{2m}} \sum_{j=0}^{2m-1} j! \binom{2m-1}{j} \frac{1}{(l\pi)^{j+1}} \sin\left(k\pi + \frac{j\pi}{2}\right) \\
&\quad - (2m-1)! \frac{1}{(2l\pi)^{2m}} \sin\left(\frac{(2m-1)\pi}{2}\right) \\
&= \frac{(-1)^l}{2^{2m}} \sum_{j=0}^{2m-1} j! \binom{2m-1}{j} \frac{1}{(l\pi)^{j+1}} \sin\left(\frac{j\pi}{2}\right) \\
&\quad - (2m-1)! \frac{(-1)^{m+1}}{(2l\pi)^{2m}} \\
&= \frac{(-1)^l}{2^{2m}} \sum_{j=1}^{m} (-1)^{j-1} \binom{2m-1}{2j-1} (2j-1)! \frac{1}{(l\pi)^{2j}} \\
&\quad + (2m-1)! \frac{(-1)^m}{(2l\pi)^{2m}},
\end{aligned}
\end{equation*}
which completes the proof.
\end{proof}

\begin{lemma} \label{lem2-thm1}
For all integers $m\geq0,$ we have
$$\lim_{x\to\frac12}\left( -(2x)^{2m}\log\cos(\pi x)+\log(1-(2x)^2) \right)=\log\left(\frac4\pi\right).$$
\end{lemma}
\begin{proof}
Observe the following algebraic manipulation for $x$ near $\frac{1}{2}$:
\begin{equation}\label{4.1}
\begin{aligned}
-(2x)^{2m}\log\cos(\pi x) &+ \log(1 - (2x)^2) \\
&= -(2x)^{2m} \log\left(\frac{\cos(\pi x)}{1 - (2x)^2}\right) + \left(1 - (2x)^{2m}\right)\log(1 - (2x)^2) \\
&= -(2x)^{2m} \log\left(\frac{\cos(\pi x)}{1 - (2x)^2}\right) + \frac{\log(1 - (2x)^2)}{\frac{1}{1 - (2x)^{2m}}}.
\end{aligned}
\end{equation}
To evaluate the limit as $x \to \frac{1}{2}$, we apply L'Hospital's Rule to each term separately and get
\begin{equation}
\lim_{x\to\frac{1}{2}} \frac{\cos(\pi x)}{1 - (2x)^2} = \frac{\pi}{4}
\quad\text{and}\quad
\lim_{x\to\frac{1}{2}} \frac{\log(1 - (2x)^2)}{\frac{1}{1 - (2x)^{2m}}} = 0,
\end{equation}
thus by (\ref{4.1}) we have
\begin{equation}
\lim_{x\to\frac{1}{2}}\left(-(2x)^{2m}\log\cos(\pi x) + \log(1 - (2x)^2)\right)= -(1)^{2m}\log\left(\frac{\pi}{4}\right) + 0 
=\log\left(\frac4{\pi}\right),
\end{equation}
which is the desired result.
\end{proof}

To prove Theorem \ref{thm2}, we need the following lemma.

\begin{lemma} \label{lem3-thm1}
For all integers $m\geq0,$ we have
$$
\int_0^{\frac12}\frac{(2x)^{m+1}-2x}{1-(2x)^2}dx
=\begin{cases}
-\frac14 H_{\frac m2} &\text{if $m$ is even}, \\
-\frac14H_{\frac{m-1}{2}}-\frac12\bar H_m +\frac12\log2 &\text{if $m$ is odd}.
\end{cases}
$$
\end{lemma}
\begin{proof}
The result follows immediately by letting $x\to 2x$ in Lemma 3.1 of \cite[p. 169]{MR} 
 \end{proof}
 
The following lemma is an analogy of \cite[Lemma 2.1]{MR} for the Dirichlet lambda function $\lambda(s)$.

\begin{lemma}\label{mr-lem}
The following assertions hold:
\begin{enumerate}
\item[\rm(1)] The series $\sum_{n=1}^\infty(\lambda(2n)-1)(2z)^{2n}$ converges for $|z|<\frac32$ to
$$\begin{cases}
\displaystyle
\frac12\pi z\tan(\pi z)-\frac{(2z)^2}{1-(2z)^2}&\text{if $|z|<\frac32,\; z\not\in\left\{-\frac12,0,\frac12 \right\}$}, \\
\displaystyle
\frac14 &\text{if } z=\pm\frac12, \\
\displaystyle
0&\text{if } z=0,
\end{cases}$$
and diverges for $|z|\geq\frac32.$
\item[\rm(2)] The series $\displaystyle\sum_{n=1}^\infty\frac{\lambda(2n)-1}{n+z}$ converges for $z\in\mathbb C\setminus \{-1,-2,-3,\ldots\}.$
\end{enumerate}
\end{lemma}
\begin{proof}
(1)
Following the methodology in \cite[Lemma 2.1, p. 165]{MR}, we begin with Proposition \ref{pro-FS}(1) which establishes
\begin{equation*}
\pi z \tan(\pi z) = \sum_{n=1}^\infty \frac{2(2z)^2}{(2n-1)^2 - (2z)^2}, \quad 2z \in \mathbb{C} \setminus (2\mathbb{Z} + 1).
\end{equation*}
For $0 < |z| < \frac{1}{2}$, leveraging absolute convergence to rearrange terms, we define
\begin{equation}\label{def-s-h}
\begin{aligned}
S(z) &:= \sum_{n=1}^\infty (\lambda(2n)-1)(2z)^{2n} \\
&= \sum_{m=1}^\infty \sum_{n=1}^\infty \left(\frac{(2z)^2}{(2m-1)^2}\right)^n - \frac{(2z)^2}{1 - (2z)^2} \\
&= \sum_{m=1}^\infty \frac{(2z)^2}{(2m-1)^2 - (2z)^2} - \frac{(2z)^2}{1 - (2z)^2} \\
&= \sum_{m=2}^\infty \frac{(2z)^2}{(2m-1)^2 - (2z)^2} =: s(z) \\
&= \frac{1}{2}\pi z \tan(\pi z) - \frac{(2z)^2}{1 - (2z)^2} =: h(z).
\end{aligned}
\end{equation}
Note that $h(z)$ is holomorphic in the domain
\[
 \left\{ z \in \mathbb{C} \mid 2z \not\in 2\mathbb{Z} + 1 \right\}.
\]
Since the points $0$ and $\pm\frac{1}{2}$ are removable singularities,
 $h(z)$ admits a holomorphic extension to $|z| < \frac{3}{2}$
with vlaues $h(0) = 0$ and $h\left(\pm\frac{1}{2}\right) = \frac{1}{4}$ (matching $s(z)$ at these points),
and so its Taylor series expansion  converges locally uniformly for $|z| < \frac{3}{2}$. Since
\[
\lim_{n \to \infty} \left( \lambda(2n) - 1 \right) 9^n = 1,
\]
the function \( S(z) \) diverges for all \( |z| \geq \frac{3}{2} \).  
To establish the above limit, we start with the expression
\[
9^n \left( \lambda(2n) - 1 \right) = \sum_{j=2}^{\infty} \left( \frac{3}{2j - 1} \right)^{2n}.
\]
For \( j = 2 \), we have \( \frac{3}{2 \cdot 2 - 1} = \frac{3}{3} = 1 \), so the first term is \( 1^{2n} = 1 \). Thus we can write
\[
9^n \left( \lambda(2n) - 1 \right) = 1 + \sum_{j=3}^{\infty} \left( \frac{3}{2j - 1} \right)^{2n}.
\]
Now consider the series \( \sum_{j=3}^{\infty} \left( \frac{3}{2j - 1} \right)^{2n} \). For \( j \geq 3 \), we have \( \left| \frac{3}{2j - 1} \right| \leq \frac{3}{5} < 1 \), and the terms satisfy
\[
\left| \left( \frac{3}{2j - 1} \right)^{2n} \right| \leq \left( \frac{3}{2j - 1} \right)^2.
\]
Since \( \sum_{j=3}^{\infty} \left( \frac{3}{2j - 1} \right)^2 \) converges (as its terms are asymptotically proportional to \( j^{-2} \)) and is independent of \( n \), the series converges uniformly in \( n \) by the Weierstrass M-test.  

Therefore, we may interchange the limit and summation:
\[
\lim_{n \to \infty} \sum_{j=3}^{\infty} \left( \frac{3}{2j - 1} \right)^{2n} = \sum_{j=3}^{\infty} \lim_{n \to \infty} \left( \frac{3}{2j - 1} \right)^{2n}.
\]
For each fixed \( j \geq 3 \), since \( \left| \frac{3}{2j - 1} \right| < 1 \), we have \( \lim_{n \to \infty} \left( \frac{3}{2j - 1} \right)^{2n} = 0 \). Consequently,
\[
\sum_{j=3}^{\infty} \lim_{n \to \infty} \left( \frac{3}{2j - 1} \right)^{2n} = \sum_{j=3}^{\infty} 0 = 0.
\]
Combining these results yields
\[
\lim_{n \to \infty} 9^n \left( \lambda(2n) - 1 \right) = \lim_{n \to \infty} \left( 1 + \sum_{j=3}^{\infty} \left( \frac{3}{2j - 1} \right)^{2n} \right) = 1 + 0 = 1,
\]
which completes the proof.
\noindent

(2)
The result follows directly from Part (1) combined with the growth estimate: for fixed $z \in \mathbb{C}$ and sufficiently large $n$,
\[
|n + z| \geq \frac{n}{2} \geq 1.
\]
This polynomial growth dominates the exponential decay in the series coefficients, ensuring convergence.
\end{proof}

\section{Proofs of the main results}\label{proofs}

In this section, we shall prove our main results which have been stated in Section~\ref{main results}.

\subsection*{Proof of Proposition \ref{pro1}}
\noindent
(1) By Lemma \ref{mr-lem}(1) in the case of $z=\frac{1}{2}$, we have
	\begin{equation}
		\sum_{n=1}^\infty (\lambda(2n)-1)= \frac{1}{4}, 
			\end{equation}
which is the desired result.
	\vspace{0.5em}
\noindent

(2) By (\ref{def-s-h}), we have
	\begin{equation}\label{5.2}
		h(x) = \sum_{n=1}^\infty (\lambda(2n)-1)(2x)^{2n} 
		= \frac{1}{2}\pi x \tan(\pi x) - \frac{(2x)^2}{1 - (2x)^2}.
	\end{equation}
In the following, we interpret the symbol $\left[ f(x) \right]_0^{\frac{1}{2}}$ as the evaluation
	\begin{equation}
		\left[ f(x) \right]_0^{\frac{1}{2}} 
		= \lim_{x \to \frac{1}{2}^-} f(x) - \lim_{x \to 0^+} f(x).
	\end{equation}
	
\noindent
Dividing (\ref{5.2}) by $x$ and integrating over the interval $[0, \frac{1}{2}]$, we have
	\begin{equation}
		\begin{aligned}
			\sum_{n=1}^\infty \frac{\lambda(2n)-1}{2n}
			&= \left[ -\frac{1}{2}\log\cos(\pi x) + \frac{1}{2}\log(1 - 4x^2) \right]_0^{\frac{1}{2}} \\
			&= \frac{1}{2} \lim_{x \to \frac{1}{2}^-} \left( -\log\cos(\pi x) + \log(1 - 4x^2) \right) \\
			&= \frac{1}{2}\log\left( \frac{4}{\pi} \right),
		\end{aligned}
	\end{equation}
which completes the proof.
\hfill $\square$

\subsection*{Proof of Theorem \ref{thm1}}
By Lemma \ref{mr-lem} (1) and multiplication by $(2x)^{2m-1}$, we have
\begin{equation*}
\sum_{n=1}^\infty (\lambda(2n)-1)(2x)^{2n+2m-1} = \frac{\pi}{4} (2x)^{2m} \tan(\pi x) - \frac{(2x)^{2m+1}}{1 - (2x)^2}.
\end{equation*}
Since the series is uniform convergence on the interval  $\left[0, \frac{1}{2}\right]$, we integrate it term-by-term and obtain
\begin{equation}\label{thm1.1}
\sum_{n=1}^\infty \frac{\lambda(2n)-1}{n+m} = \int_0^{\frac{1}{2}} \left( \pi(2x)^{2m}\tan(\pi x) - \frac{4(2x)^{2m+1}}{1 - (2x)^2} \right) dx.
\end{equation}
While $\int_0^{\frac{1}{2}} \pi(2x)^{2m}\tan(\pi x)dx$ is divergent, we compute its antiderivative via integration by parts:
\begin{equation}\label{int-tan}
\int \pi(2x)^{2m}\tan(\pi x)dx = -(2x)^{2m}\log\cos(\pi x) + m2^{2m} \int (2x)^{2m-1}\log\cos(\pi x)dx.
\end{equation}
Recall the Fourier series for $\log\cos(\pi x)$ from \cite[1.441(4), p. 46]{MR}:
\begin{equation*}
\log\cos(\pi x) = -\log 2 + \sum_{k=1}^\infty (-1)^{k-1} \frac{\cos(2k\pi x)}{k} \quad \text{for } |x| < \frac{1}{2}.
\end{equation*}
Substituting this into (\ref{int-tan}) yields
\begin{equation}\label{int-tan-se}
\begin{aligned}
\int \pi(2x)^{2m}\tan(\pi x)dx &= -(2x)^{2m}\log\cos(\pi x) - (\log 2)(2x)^{2m} \\
&\quad + m2^{2m+1} \sum_{k=1}^\infty \frac{(-1)^{k-1}}{k} \int x^{2m-1}\cos(2k\pi x)dx.
\end{aligned}
\end{equation}
Since 
\begin{equation}\label{int-frac}
\frac{4(2x)^{2m+1}}{1 - (2x)^2} = -4\sum_{k=0}^{m-1} (2x)^{2k+1} + \frac{8x}{1 - (2x)^2},
\end{equation}
we have
\begin{equation*}
\int \frac{4(2x)^{2m+1}}{1 - (2x)^2}dx = -\sum_{k=0}^{m-1} \frac{(2x)^{2k+2}}{k+1} - \log(1 - (2x)^2) + C,
\end{equation*}
where $C$ is a constant of integration.
Denote the right hand side of (\ref{thm1.1}) by
 $$I(m) := \int_0^{\frac{1}{2}} \left( \pi(2x)^{2m}\tan(\pi x) - \frac{4(2x)^{2m+1}}{1 - (2x)^2} \right) dx.$$
Then by combining (\ref{int-tan}), (\ref{int-tan-se}) and (\ref{int-frac}), we have
\begin{equation*}
\begin{aligned}
I(m) &= \left[ -(2x)^{2m}\log\cos(\pi x) + \log(1 - (2x)^2) - (\log 2)(2x)^{2m} + \sum_{k=0}^{m-1} \frac{(2x)^{2k+2}}{k+1} \right]_0^{\frac{1}{2}} \\
&\quad + m2^{2m+1} \sum_{k=1}^\infty \frac{(-1)^{k-1}}{k} \int_0^{\frac{1}{2}} x^{2m-1}\cos(2k\pi x)dx.
\end{aligned}
\end{equation*}
Finally,  applying Lemmas \ref{lem1-thm1} and \ref{lem2-thm1}, we obtain
\begin{equation*}
\begin{aligned}
I(m) &= \log\left(\frac{4}{\pi}\right) - \log 2 + H_m + \frac{(-1)^m (2m)!}{\pi^{2m}} (\zeta_E(2m+1) + \zeta(2m+1)) \\
&\quad + 2m \sum_{j=1}^{m-1} (-1)^j \binom{2m-1}{2j-1} (2j-1)! \frac{1}{\pi^{2j}} \zeta(2j+1),
\end{aligned}
\end{equation*}
where $H_m$ denotes the $m$-th harmonic number. This matches the desired result through identities (\ref{zeta-lam1}) and (\ref{zeta-lam2}).
\hfill $\square$

\subsection*{Proof of Theorem \ref{thm2}}
Using Lemma \ref{mr-lem}(1) and multiplying by $(2x)^{m-1}$, we need to evaluate the series
\begin{equation}\label{ser-2n+m}
\sum_{n=1}^\infty (\lambda(2n) - 1)(2x)^{2n+m-1} = \frac{\pi}{4} (2x)^m \tan(\pi x) - \frac{(2x)^{m+1}}{1 - (2x)^2}.
\end{equation}

The remainder of the proof parallels that of Theorem \ref{thm1}, but requires the following integral identity \cite[Theorem 2.3]{SK2025}:
\begin{equation}
\begin{aligned}
m \int_0^{\frac{1}{2}} x^{m-1} \log \cos(\pi x)  dx 
&= -\frac{\log 2}{2^m} - \frac{m}{2^m} \sum_{j=0}^{m-1} j! \binom{m-1}{j} \frac{\sin\left(\frac{j\pi}{2}\right)}{\pi^{j+1}} \zeta(j+2) \\
&\quad - \frac{m!}{(2\pi)^m} \zeta_E(m+1) \cdot 
\begin{cases} 
0 & \text{if $m$ is odd}, \\ 
(-1)^{\frac{m}{2}+1} & \text{if $m$ is even}.
\end{cases}
\end{aligned}
\end{equation}

By applying Lemma \ref{lem2-thm1}, Lemma \ref{lem3-thm1}, and the above integral, we obtain
\begin{equation}
\begin{aligned}
\sum_{n=1}^\infty \frac{\lambda(2n) - 1}{2n + m}
&= \frac{1}{2} \int_0^{\frac{1}{2}} \left( \pi (2x)^m \tan(\pi x) - \frac{4(2x)^{m+1}}{1 - (2x)^2} \right) dx \\
&= \frac{1}{2} \left[ -(2x)^m \log \cos(\pi x) + \log(1 - (2x)^2) \right]_0^{\frac{1}{2}} \\
&\quad + m 2^{m-1} \int_0^{\frac{1}{2}} x^{m-1} \log \cos(\pi x)  dx - 2 \int_0^{\frac{1}{2}} \frac{(2x)^{m+1} - 2x}{1 - (2x)^2} dx \\
&= \frac{1}{2} \log\left(\frac{2}{\pi}\right) + \frac{m}{2} \sum_{j=1}^{\left\lceil \frac{m-1}{2} \right\rceil} (-1)^j \binom{m-1}{2j-1} \frac{(2j-1)!}{\pi^{2j}} \zeta(2j+1) \\
&\quad - \frac{m!}{2\pi^{m}} \zeta_E(m+1) \cdot 
\begin{cases} 
(-1)^{\frac{m}{2}+1} & \text{if $m$ is even} \\ 
0 & \text{if $m$ is odd}
\end{cases} \\
&\quad + 
\begin{cases} 
\frac{1}{2} H_{\frac{m}{2}} & \text{if $m$ is even}, \\ 
\frac{1}{2} H_{\frac{m-1}{2}} - \log 2 + \bar{H}_m & \text{if $m$ is odd}.
\end{cases}
\end{aligned}
\end{equation}
The result then follows by applying identity \eqref{zeta-lam1}.
\hfill $\square$
\subsection*{Proof of Corollary \ref{cor-SC-1}}
It is shown in \cite[p. 164, Theorem 1.2]{MR}  (see also (\ref{3})) that
\begin{equation}\label{sum1}
\begin{aligned}
\sum_{n=1}^{\infty} \frac{\zeta(2n)-1}{n+m} 
&= \frac{1}{2m}+H_m-\log\pi \\
&\quad-2m\sum_{j=1}^{m-1}(-1)^{j+1}\binom{2m-1}{2j-1}\frac{(2j-1)!}{(2\pi)^{2j}}\zeta(2j+1).
\end{aligned}
\end{equation}
Since
$$\sum_{n=1}^{\infty} \frac{\zeta(2n)}{2^{2n}(n+m)} 
=\sum_{n=1}^\infty\frac{\zeta(2n)-1}{n+m}-\sum_{n=1}^\infty\frac{\lambda(2n)-1}{n+m},$$
by (\ref{zeta-lam1}), Theorem \ref{thm1} and (\ref{sum1}), we obtain the identity
\begin{equation}
\begin{aligned}
\sum_{n=1}^{\infty} \frac{\zeta(2n)}{2^{2n}(n+m)} 
&=\frac1{2m}-\log2-\frac{(-1)^{m}2(2m)!}{\pi^{2m}}\lambda(2m+1) \\
&\quad+2m\sum_{j=1}^{m-1}(-1)^j\binom{2m-1}{2j-1}\frac{(2j-1)!}{(2\pi)^{2j}}\zeta(2j+1) \\
&\quad-2m\sum_{j=1}^{m-1}(-1)^j\binom{2m-1}{2j-1}\frac{(2j-1)!2^{2j+1}}{\pi^{2j}(2^{2j+1}-1)}\lambda(2j+1) \\
&=\frac1{2m}-\log2-\frac{(-1)^{m}(2m)!(2^{2m+1}-1)}{(2\pi)^{2m}}\zeta(2m+1)  \\
&\quad+\sum_{j=1}^{m-1}(-1)^j\binom{2m}{2j}\frac{(2j)!}{(2\pi)^{2j}}(1-2^{2j})\zeta(2j+1) 
\end{aligned}
\end{equation}
for $m\geq2.$ 
This completes the proof.  
\hfill$\square$

\subsection*{Proof of Corollary \ref{cor-SC-2}}
The proof follows the same approach as Corollary \ref{cor-SC-1}. That is, by combining equation \eqref{zeta-lam1}, Theorem \ref{thm2}, and \cite[Theorem 3.2, p. 170]{MR} (which establishes \eqref{3-1}), we obtain the stated result.
\hfill $\square$

\subsection*{Proof of Theorem \ref{thm3}}
 Setting \( z = w + \frac{1}{2} \) (equivalently, \( w = z - \frac{1}{2} \)), we have
\[
\begin{aligned}
f(z) &= \frac{\pi}{2} z \tan(\pi z) - \frac{(2z)^2 - 1 + 1}{1 - (2z)^2} \\
&= 1 + \frac{\pi}{2} z \tan(\pi z) - \frac{1}{2}\left( \frac{1}{1 - 2z} + \frac{1}{1 + 2z} \right) \\
&= 1 - \frac{\pi}{2} \left( w + \frac{1}{2} \right) \cot(\pi w) - \frac{1}{2} \left(- \frac{1}{2w} + \frac{1}{2 + 2w} \right) =: h(w).
\end{aligned}
\]
Using the series expansion 
\begin{equation*}
\pi w \cot(\pi w) = 1 - 2\sum_{n=1}^\infty \zeta(2n) w^{2n}
\end{equation*}
for \( 0 < |w| < 1 \)  (see \cite[p. 937]{MR2}), we expand \( h(w) \) as follows
\[
\begin{aligned}
h(w) &= 1 - \frac{1}{2} \left( 1 - 2\sum_{n=1}^\infty \zeta(2n) w^{2n} \right) - \frac{1}{4w} \left( 1 - 2\sum_{n=1}^\infty \zeta(2n) w^{2n} \right) \\
&\quad + \frac{1}{4w} - \frac{1}{4(1 + w)} \\
&= \frac{1}{4} + \sum_{n=1}^\infty \zeta(2n) w^{2n} + \frac{1}{2} \sum_{n=1}^\infty \zeta(2n) w^{2n - 1} - \frac{1}{4(1 + w)}.
\end{aligned}
\]
Applying identity (\ref{zeta-lam1}), this becomes to
\[
\begin{aligned}
h(w) &= \frac{1}{4} + \sum_{n=1}^\infty \frac{2^{2n}}{2^{2n} - 1} \lambda(2n) \left( w^{2n} + \frac{1}{2} w^{2n - 1} \right) - \frac{1}{4} \sum_{n=1}^\infty (-1)^n w^n.
\end{aligned}
\]
Finally, by reversing the substitution \( w = z - \frac{1}{2} \), we have
$$
\begin{aligned}
			c_0:=\frac14, \quad c_n
			&:=\begin{cases}
				\displaystyle
				\frac12\left(\frac{2^{n+1}}{2^{n+1}-1}\right)\lambda(n+1)+\frac1{4} &\text{if $n$ is odd} \\
				\displaystyle
				\left(\frac{2^{n}}{2^{n}-1}\right)\lambda(n)-\frac1{4} &\text{if $n$ is even}
			 \end{cases}  \\
			&=\frac1{2^{\delta_n}}\left(
			\frac{2^{2\left\lceil \frac{n}{2}\right\rceil}}{2^{2\left\lceil \frac{n}{2}\right\rceil}-1} \right)
			\lambda\left(2\left\lceil \frac{n}{2}\right\rceil\right)
			-\frac{(-1)^n}{4},
\end{aligned}
$$
where 
$$\delta_n=\frac{1-(-1)^n}{2}.$$
This completes the proof.  
 \hfill$\square$
 
 \subsection*{Proof of Theorem \ref{thm4}}
This result follows by combining Lemma \ref{mr-lem} with Theorem \ref{thm3}. Consider the function defined in \eqref{ma-s-re}:
\begin{equation*}
f(z) = \sum_{n=0}^\infty c_n \left(z - \frac{1}{2}\right)^n, \quad \text{for } \left|z - \frac{1}{2}\right| < 1.
\end{equation*}
From  Lemma \ref{mr-lem}, we have its alternative representation
\[
f(z) = \sum_{n=1}^\infty (\lambda(2n) - 1)(2z)^{2n}, \quad \text{valid for } |z| < \frac{3}{2}.
\]
Then by differentiating $p$ times and evaluating the result at $z = \frac{1}{2}$, we obtain
\[
\begin{aligned}
f^{(p)}\left(\frac{1}{2}\right) 
&= \left. \sum_{n=1}^\infty \binom{2n}{p} p! \, 2^p (\lambda(2n) - 1) (2z)^{2n-p} \right|_{z=\frac{1}{2}} \\
&= p! \, c_p.
\end{aligned}
\]
This implies
\begin{equation}\label{3.11+}
\sum_{n=1}^\infty \binom{2n}{p} (\lambda(2n) - 1) = \frac{c_p}{2^p}.
\end{equation}
Finally, by applying Theorem \ref{thm3}, we get
\[
\sum_{n=1}^\infty \binom{2n}{p} (\lambda(2n) - 1) 
= \frac{1}{2^{2\left\lceil \frac{p}{2}\right\rceil} - 1} \lambda\left(2\left\lceil \frac{p}{2}\right\rceil\right) - \frac{(-1)^p}{2^{p+2}},
\]
which  completes our proof.  
\hfill$\square$

\subsection*{Proof of Corollary \ref{coro2}}
In Proposition 3.1 of \cite{MR2}, the authors have established the following closed-form expression
\begin{equation}\label{coro2.1}
\sum_{n=1}^\infty \binom{2n}{p}(\zeta(2n)-1) = 
\begin{cases}
\displaystyle
\zeta(p+1) + \frac{1}{2^{p+2}} & \text{if } p \text{ is odd}, \\
\displaystyle
\zeta(p) - \frac{1}{2^{p+2}} & \text{if } p \text{ is even}.
\end{cases}
\end{equation}
By (\ref{zeta-lam1}), we have
\begin{equation}\label{coro2.2}
\lambda(2n) = \left(1 - \frac{1}{2^{2n}}\right)\zeta(2n).
\end{equation} 
Substituting (\ref{coro2.2}) into  (\ref{3.11+}),  we have
$$\sum_{n=1}^\infty \binom{2n}{p}\frac{\zeta(2n)}{2^{2n}} 
= \sum_{n=1}^\infty \binom{2n}{p}(\zeta(2n)-1) - \frac{c_p}{2^p},$$
where $c_{p}$ are the coefficients in Theorem \ref{thm3}.
Then substituting  (\ref{coro2.1}) into the above equation, we further get that
\begin{equation}\label{coro2+}
\sum_{n=1}^\infty \binom{2n}{p}\frac{\zeta(2n)}{2^{2n}}= 
\begin{cases}
\displaystyle
\zeta(p+1) + \frac{1}{2^{p+2}}- \frac{c_p}{2^p}& \text{if } p \text{ is odd}, \\
\displaystyle
\zeta(p) - \frac{1}{2^{p+2}}- \frac{c_p}{2^p} & \text{if } p \text{ is even}.
\end{cases}
\end{equation}
In addition, by substituting  (\ref{coro2.2}) into (\ref{thm3.1}), we have the representation 
\begin{equation}\label{thm3.1+}
		\begin{aligned}
		    c_p
			&:=\frac1{2^{\delta_p}}\left(
			\frac{2^{2\left\lceil \frac{p}{2}\right\rceil}}{2^{2\left\lceil \frac{p}{2}\right\rceil}-1} \right)
			\lambda\left(2\left\lceil \frac{p}{2}\right\rceil\right)
			-\frac{(-1)^p}{4}\\
			&=\begin{cases}
			 \displaystyle
				\frac12\zeta(p+1)+\frac1{4} &\text{if $p$ is odd} \\
				\displaystyle
				\zeta(p)-\frac1{4} &\text{if $p$ is even}.
			 \end{cases} 
		\end{aligned}
	\end{equation}
Finally, by substituting (\ref{thm3.1+}) into (\ref{coro2+}), we get
\[
\begin{aligned}
\sum_{n=1}^\infty \binom{2n}{p}\frac{\zeta(2n)}{2^{2n}} 
&= \begin{cases}
\displaystyle
(1 - 2^{-(p+1)})\zeta(p+1) & \text{if } p \text{ is odd} \\
\displaystyle
(1 - 2^{-p})\zeta(p) & \text{if } p \text{ is even}
\end{cases} \\
&= \left(1 - \frac{1}{2^{2\left\lceil \frac{p}{2}\right\rceil}}\right) \zeta\left(2\left\lceil \frac{p}{2}\right\rceil\right),
\end{aligned}
\]
where the final identity follows from the ceiling function representation of parity conditions.
\hfill$\square$

\subsection*{Proof of Theorem \ref{thm5}}
Recall that for an arbitrary infinitely differentiable function \( f \), the differential operator identity holds (see \cite[p. 126, (9.42)]{QG}):
\begin{equation}\label{dif-1}
\left(z\frac{d}{dz}\right)^p f(z) = \sum_{j=1}^p {p \brace j} z^j \frac{d^j}{dz^j}f(z),
\end{equation}
where $p\in\mathbb N$ and \({p \brace j}\) denotes Stirling numbers of the second kind.
Using equation (\ref{dif-1}), the proof follows an argument analogous to that of Theorem \ref{thm4}.

\noindent
(1) Applying \eqref{dif-1} to the Taylor expansion 
\begin{equation*} 
f(z) = \sum_{n=0}^\infty c_n \left(z - \frac{1}{2}\right)^n 
\end{equation*}
 within its radius of convergence \( |z - \frac{1}{2}| < 1 \), we obtain
\[
\left.\left(z\frac{d}{dz}\right)^p f(z)\right|_{z=\frac{1}{2}} 
= \sum_{j=1}^p {p \brace j} \left(\frac{1}{2}\right)^j j! c_j.
\]

\noindent
(2) Consider the function 
\[ f(z) = \sum_{n=1}^\infty (\lambda(2n)-1)(2z)^{2n} \] 
for \( |z| < \frac{3}{2} \). Applying the operator $\left(z\frac{d}{dz}\right)^p$ again, we get
\[
\begin{aligned}
\left.\left(z\frac{d}{dz}\right)^p f(z)\right|_{z=\frac{1}{2}} 
&= \left.\left(2z\frac{d}{d(2z)}\right)^p \sum_{n=1}^\infty (\lambda(2n)-1)(2z)^{2n} \right|_{z=\frac{1}{2}} \\
&= \left.\sum_{n=1}^\infty (2n)^p (\lambda(2n)-1)(2z)^{2n} \right|_{z=\frac{1}{2}} \\
&= 2^p \sum_{n=1}^\infty n^p (\lambda(2n)-1).
\end{aligned}
\]
Combining the results from (1) and (2), we establish that
\begin{equation}\label{13-re}
\sum_{n=1}^\infty n^p (\lambda(2n)-1) = \frac{1}{2^p} \sum_{j=1}^p {p \brace j} \left(\frac{1}{2}\right)^j j!c_j.
\end{equation}
The result follows by expressing \( c_j \) via its lambda-function representation (see Theorem \ref{thm3}).
\hfill$\square$

\subsection*{Proof of Corollary \ref{coro3}}
Let
\begin{equation}\label{coro31}
g(z) := \sum_{n=1}^\infty (\zeta(2n) - 1) z^{2n} = 
\begin{cases} 
\frac{1}{2}(1 - \pi z \cot(\pi z)) & \text{if } |z| < 2,\ z \not\in \{-1,0,1\}, \\ 
\frac{3}{4} & \text{if } z = 0, \\
0 & \text{if } z = \pm 1. 
\end{cases}
\end{equation}
It admits the following Taylor expansion about $z=1$ for $|z-1|<1$ \cite[p.~930, (0.1) and p.~937, Theorem~2.1]{MR2}:
\[
g(z) = \sum_{n=0}^\infty a_n (z-1)^n,
\]
where the coefficients are given by
\begin{equation}\label{co-an}
a_0 := \frac{3}{4}, \quad 
a_n := 
\begin{cases} 
\zeta(n+1) + \frac{1}{2^{n+2}} & \text{if $n$ is odd}, \\ 
\zeta(n) - \frac{1}{2^{n+2}} & \text{if $n$ is even},
\end{cases}
\end{equation}
and satisfy $g^{(n)}(1) = n! a_n$ (see Remark \ref{rem-g}).

Following the methodology of the proof for Theorem \ref{thm5}, we obtain 
\begin{equation}\label{15-re}
\sum_{n=1}^\infty n^p (\zeta(2n) - 1) = \frac{1}{2^p} \sum_{j=1}^p {p \brace j} j! a_j
\end{equation}
for $p \in \mathbb{N}$. This result is equivalent to Mortini and Rupp's Theorem 2.3 \cite[p.~938]{MR2} from elementary transformations.

Using successively \eqref{coro2.2}, \eqref{13-re},  \eqref{15-re}, \eqref{co-an}, and \eqref{thm3.1+}, we have
\[
\begin{aligned}
\sum_{n=1}^\infty n^p \frac{\zeta(2n)}{2^{2n}} 
&= \sum_{n=1}^\infty n^p (\zeta(2n) - 1) - \sum_{n=1}^\infty n^p (\lambda(2n) - 1) \\
&= \frac{1}{2^p} \sum_{j=1}^p {p \brace j} j! \left( a_j - \left( \frac{1}{2} \right)^j c_j \right) \\
&= \frac{1}{2^p} \sum_{j=1}^p {p \brace j} j! 
\begin{cases} 
\left(1 - \frac{1}{2^{j+1}}\right) \zeta(j+1) & \text{if $j$ is odd} \\ 
\left(1 - \frac{1}{2^{j}}\right) \zeta(j) & \text{if $j$ is even}
\end{cases} \\
&= \frac{1}{2^p} \sum_{j=1}^p {p \brace j} j! 
\left(1 - \frac{1}{2^{2\lceil j/2 \rceil}}\right) \zeta\left(2\left\lceil \frac{j}{2} \right\rceil\right).
\end{aligned}
\]
This completes the proof of Corollary \ref{coro3}. \hfill $\square$

\section*{Acknowledgements}
We thank Professor Christophe Vignat for his interest in this work and for his valuable comments and suggestions. We also acknowledge the use of \texttt{Mathematica} and \texttt{Wolfram Alpha} (\url{https://www.wolframalpha.com}) for verifying the analytical results presented in this work.

\bibliography{central}

\end{document}